\documentclass{article}
\usepackage{tikz}

\usepackage{mathtools}
\usepackage{amsmath}
\usepackage{amsfonts}
\usepackage{amssymb}
\usepackage{amsthm}
\usepackage{mathrsfs}
\usepackage{manfnt}
\usepackage{xcolor}
\usepackage{tcolorbox}
\usepackage{hyperref}
\usepackage{enumitem}
\usepackage{upgreek}
\usepackage{scalerel}
\usepackage{graphicx,stackengine}

\theoremstyle{plain}
\newtheorem{theorem}{Theorem}[section]
\theoremstyle{definition}
\newtheorem{definition}[theorem]{Definition}

\theoremstyle{plain}
\newtheorem{lemma}[theorem]{Lemma}
\newtheorem{proposition}[theorem]{Proposition}
\newtheorem{corollary}[theorem]{Corollary}

\theoremstyle{remark}
\newtheorem{propparagraph}[theorem]{}
\newtheorem{defparagraph}[theorem]{}

\theoremstyle{remark}

\newtheorem{remark}[theorem]{Remark}
\theoremstyle{plain}
\newtheorem{introTheorem}{Theorem}

\newcommand{\weight}{\mathbin{\vrule height 1.6ex depth 0pt width
		0.14ex\vrule height 0.14ex depth 0pt width 1.3ex}}
\DeclareMathOperator{\Tan}{Tan}

\DeclareMathOperator{\G}{{\bf G}}
\DeclareMathOperator{\Var}{{\bf V}}

\DeclareMathOperator{\IVar}{{\bf IV}}

\DeclareMathOperator{\dist}{dist}

\DeclareMathOperator{\Lip}{Lip}
\DeclareMathOperator{\spt}{spt}
\DeclareMathOperator{\D}{D}

\DeclareMathOperator*{\Clos}{Clos}

\DeclareMathOperator*{\bfalpha}{\boldsymbol{\alpha}}

\DeclareMathOperator*{\Density}{\boldsymbol{\rm \Theta}}
\DeclareMathOperator{\vVarOp}{{\rm\bf v}}

\DeclareMathOperator{\VarTan}{VarTan}

\DeclareMathOperator{\bfp}{{\bf p}}

\DeclareMathOperator{\bfq}{{\bf q}}

\DeclareMathOperator{\bfh}{{\bf h}}

\DeclareMathOperator{\GDensity}{\Theta}

\DeclareMathOperator{\grad}{grad}

\author{Yu Tong Liu and Myles Workman}
\title{Parabolic rectifiability of the Brakke flow}
\date{}

\begin{document}

\maketitle

\begin{abstract}
    We prove that the support of the canonical space-time measure for a Brakke flow is a parabolic $(k+2)$-rectifiable set. 
    As a consequence, we obtain that at almost all points along the flow, with respect to this canonical space-time measure, there exists a unique, static, planar tangent flow, and that various notions of density for the flow agree at these points. 

    Moreover, following on from our previous work \cite{LW-Space-Time-GrassmannMeasureBrakkeFlow}, we continue to develop the approach to the Brakke flow as a space-time-Grassmann measure. 
    We prove that the standard notion of convergence for Brakke flows, coming from \cite[7.1]{Ilmanen-EllipticRegularization}, is equivalent to the convergence of these space-time-Grassmann Radon measures.
    This gives an alternate notion of varifold convergence to the one exhibited in \cite[7.1(ii)]{Ilmanen-EllipticRegularization}.
\end{abstract}

\begin{tabular}{ l p{9cm} }
    \textbf{MSC 2020} & 53E10, 28A75, 49Q20 \\
    \textbf{Keywords} & Brakke flow, space-time-Grassmann measure, parabolic rectifiability
\end{tabular}

\bigskip

Let $J \subset\mathbf{R}$ be a non-empty open interval, and $U \subset \mathbf{R}^{n}$ be a non-empty open set, then a $k$-dimensional Brakke flow over $J \times U$, is a family of Radon measures $\{ \mu (t) \, \colon \, t \in J \}$ over $U$, which evolve as a weak, measure theoretic solution to the mean curvature flow. 
Recall (\cite{LW-Space-Time-GrassmannMeasureBrakkeFlow}) that for a Brakke flow, there exists a canonical Radon measure $\| V \|$ on $J \times U$, given by, 
\begin{equation*}
    \textstyle \| V \| (\psi) = \int_{J} \mu (t)_{x} (\psi (t, x)) \, d \mathscr{L}^{1}_{t},
\end{equation*}
for every continuous real valued function $\psi$ on $J \times U$ with compact support.
Our main result concerns the regularity of this measure $\| V \|$.
\begin{introTheorem}\label{introthm: parabolic rectifiability}
    Suppose $\{ \mu (t) \, \colon \, t \in J \}$ is a $k$-dimensional Brakke flow over $J \times U$, and $\| V \|$ is its canonical space-time measure, then $\spt  \| V \|$ is a (vertically) parabolic $(k + 2)$-rectifiable set.
\end{introTheorem}
The notion of parabolic rectifiability has been recently developed in \cite{mattila2021parabolicrectifiabilitytangentplanes}. 
For the reader's convenience we collect relevant definitions and results from \cite{mattila2021parabolicrectifiabilitytangentplanes} in \ref{subsec: parabolic rectifiability}.

As a consequence to \textbf{Theorem \ref{introthm: parabolic rectifiability}} we obtain that at $\| V \|$ almost all points in $J \times U$, there exists a unique, static, planar tangent flow, and, at such points, all various notions of density for Brakke flow are equal.
Moreover, one may then write our measure $\| V \|$, simply as the restriction of the $(k + 2)$-dimensional parabolic Hausdorff measure to the $(k + 2)$-dimensional parabolic density function.

\begin{introTheorem}\label{introthm: existence of static tangent flow and equality of densities}
    Suppose $\{ \mu (s) \, \colon \, s \in J \}$ is a $k$-dimensional Brakke flow over $J \times U$, and $\| V \|$ is its canonical space-time measure. 
    Then for $\|V\|$ almost every $(t,x)\in J\times U$, $\{ \mu (s) \, \colon \, s \in J \}$ has a unique static planar tangent flow at $(t, x)$, given by
        \begin{equation*}
                \textstyle \{ \Density^{k}(\mu (t), x) \, \cdot \, \mathscr{H}^{k} \weight \Tan^k(\mu (t),x) \, \colon \, s \in \mathbf{R} \}.
        \end{equation*} 
        and moreover, 
        \begin{align*}
            \Theta (V, t, x) = \mathbf{\Theta}^{k} (\mu (t), x) =& \, (2 \boldsymbol{\alpha} (k))^{-1} \Theta^{k + 2}_{\rho} (\| V \|, t, x), \\ 
            =& \, \boldsymbol{\beta} (k)^{-1} \Theta^{k + 2}_{d} (\| V \|, t, x).
        \end{align*}
        Furthermore, we have that 
        \begin{equation*}
        \| V \| = \mathscr{H}^{k + 2}_{\rho} \weight \boldsymbol{\nu} (k)^{-1} \Theta_{\rho}^{k + 2} (\| V \|, \, \cdot \, ) = \mathscr{H}^{k + 2}_{d} \weight \boldsymbol{\omega} (k)^{-1} \Theta_{d}^{k + 2} (\| V \|, \, \cdot \, ).
    \end{equation*}
\end{introTheorem}

Here, $\Theta (V, t ,x)$ denotes the Gaussian density, and the other densities are defined in \ref{para: densities}, with the latter two being with respect to the parabolic metrics $\rho$ and $d$, defined in \ref{para: parabolic metrics}.
The constant $\boldsymbol{\alpha} (k)$ is defined in \ref{para: hausdorff measure}, and the constants $\boldsymbol{\beta} (k)$, $\boldsymbol{\nu} (k)$, and $\boldsymbol{\omega} (k)$ are defined in \ref{para: parabolic metrics}. 

We emphasise that the results in \textbf{Theorem \ref{introthm: existence of static tangent flow and equality of densities}} were previously unknown.
What was previously known, by stratification \cite{White1997}, was that for $\mathscr{H}^{k + 2}_{\rho}$ almost all $(t, x) \in \mathrm{spt} \, \| V \|$, there exists a static, planar tangent flow at $(t, x)$. 
The stratification however does not assert that this static planar tangent flow is unique, and thus does not yield the equality of densities.

Before outlining the strategy of the proof, we first comment on our approach to the Brakke flow. 
As a follow-on from our previous work \cite{LW-Space-Time-GrassmannMeasureBrakkeFlow}, throughout much of this paper we opt to define the Brakke flow as a space-time-Grassmann measure $V$, and work exclusively with $V$ and $\| V \|$, instead of the usual approach of working with a $1$ parameter family of Radon measures over $U$ (or the equivalent $k$-dimensional varifolds).
One motivation is that these measures satisfy better properties than the individual time slice measures $\mu (t)$. 
For example $\| V \|$ satisfies uniform lower $(k + 2)$-dimensional density bounds, see \ref{prop: density bounds}.
Another motivation is that this allows us to apply classical techniques from measure theory in this space-time setting.
For example, in this present paper we make use of the notion of approximate continuity in space-time, as well as the differentiation theory for space-time Radon measures. 
A further benefit of this approach is that these measures simplify the convergence theory of Brakke flows.

\begin{introTheorem}[Convergence of Brakke Flows]\label{introprop: convergence of Brakke flows} 
    Let $V_{1}, \, V_{2}, \, \ldots$, and $V$ be space-time-Grassmann Brakke flows over $J \times \mathbf{G}_{k} (U)$ (see \ref{def: space-time-Grassmann Brakke flow}), then the following three statements are equivalent: 
    \begin{enumerate}[label=\ref{introprop: convergence of Brakke flows}.\arabic{enumi}]
         \item \label{introprop: convergence of Brakke flows:1} $V = \lim\limits_{i\to\infty} V_i$. 
        \item \label{introprop: convergence of Brakke flows:2} $\|V\| = \lim\limits_{i\to\infty} \|V_i\|$.
        \item \label{introprop: convergence of Brakke flows:3} 
        There exists a countable set $D \subset J$ such that for every representative $\mu_i$ of $V_{i}$ (see \ref{proppara: def of representative Brakke Flow}), and every $t\in J \sim D$,
        \[
            \| \langle V, t \rangle \| = \lim\limits_{i\to\infty} \mu_i(t).
        \]
    \end{enumerate}
\end{introTheorem}

The standard notion of convergence of Brakke flows, coming from \cite[7.1]{Ilmanen-EllipticRegularization}, is given by \ref{introprop: convergence of Brakke flows:3}, after potentially taking a further subsequence to obtain convergence along the countable set of times $D$.
Thus a key point of interest in \textbf{Theorem \ref{introprop: convergence of Brakke flows}}, is that \ref{introprop: convergence of Brakke flows:1} gives an alternate notion of varifold convergence to the one exhibited in \cite[7.1(ii)]{Ilmanen-EllipticRegularization}.

We now remark on the proofs of \textbf{Theorem \ref{introthm: parabolic rectifiability}} and \textbf{Theorem \ref{introthm: existence of static tangent flow and equality of densities}}.
First, it may be shown that at point $(t, x) \in \mathrm{spt} \, \| V \|$, at which $T(s, y) = \mathrm{Tan}^{k} ( \langle V, s \rangle, y)$ is approximately continuous with respect to $\| V \|$, any tangent flow at $(t, x)$, must be comprised at each time, of an at most countable collection of planes parallel to $T (t, x)$, each with potentially time dependent multiplicities.
The self-shrinking nature of these tangent flows then implies that we must in fact just have $T (t, x)$, with constant density for negative times, given by the Gaussian density of $V$ at $(t, x)$. 
The issue now to concluding \textbf{Theorem \ref{introthm: existence of static tangent flow and equality of densities}} is that this density could drop for non-negative times. 
However, $\mathrm{spt} \, \| V \|$ being a parabolic rectifiable set (\textbf{Theorem \ref{introthm: parabolic rectifiability}}) rules out this behaviour almost everywhere.
Indeed once one has that $\mathrm{spt} \, \|V\|$ is a (vertical) parabolic $(k + 2)$-rectifiable set, then almost everywhere, tangent measures to $\mathscr{H}^{k + 2}_{\rho} \weight \mathrm{spt} \, \| V \|$ must have the form $c \cdot \mathscr{H}^{k + 2}_{\rho} \weight \mathbf{R} \times T$, for $T \in \mathbf{G} (n, k)$, and $0 < c < +\infty$. 
Thus, as $\| V \|$ is absolutely continuous with respect to $\mathscr{H}^{k + 2}_{\rho} \weight \mathrm{spt} \, \| V \|$, tangent measures to $\| V \|$ almost everywhere must have this same form, implying that for our tangent flows, the time dependent multiplicity function cannot drop for non-negative times.
To prove \textbf{Theorem \ref{introthm: parabolic rectifiability}}, the previous discussion, along with mean value inequality for Brakke flows (or equivalently Brakke's clearing out lemma) implies that, locally about $(t, x)$, $\mathrm{spt} \, \| V \|$ must be contained in a parabolic cone about $\mathbf{R} \times T (t, x)$, with arbitrarily small cone angle.
It then follows, via classical arguments, that $\mathrm{spt} \, \| V \|$ is a (vertical) parabolic $(k + 2)$-rectifiable set.

We now outline the structure of the paper. 
For the reader's convenience in section \ref{subsec: notation} we collect notations, and in section \ref{subsec: parabolic rectifiability} we collect relevant definitions and results from \cite{mattila2021parabolicrectifiabilitytangentplanes}.
In \ref{subsec: Brakke Flow preliminaries} we collect preliminaries on the Brakke flow, and rephrase many classical results (monotonicity formula, mean value inequality etc.) in terms of the space-time-Grassmann measure $V$. 
We prove \textbf{Theorem \ref{introprop: convergence of Brakke flows}} in section \ref{sec: convergence of Brakke Flows}, and \textbf{Theorem \ref{introthm: parabolic rectifiability}} and \textbf{Theorem \ref{introthm: existence of static tangent flow and equality of densities}} in section \ref{sec: parabolic rectifiability}.

\subsection*{Acknowledgments} 

Both authors are grateful to Prof. Ulrich Menne for many valuable discussions and for sharing ideas that contributed to this work.

\section{Preliminaries}

\subsection{Notation}\label{subsec: notation}

Throughout we have $U$ being an open subset of $\mathbf{R}^{n}$, and $J$ is an open interval of $\mathbf{R}$, and $\mathscr{P}$ denoting the set of positive integers.

\begin{propparagraph}[Upper derivative]
    For a real valued function $f$ on $J$, 
    and $t\in J$,
    we define 
\[ \overline{\D}_sf(s)|_t = \limsup\limits_{s\to t} \frac{f(s)-f(t)}{s-t}\]
\end{propparagraph}

\begin{propparagraph}[Radon measures and continuous functions]
    For a locally compact Hausdorff space $X$, $\mathscr{K} (X)$ denotes the space of compactly supported continuous functions on $X$, and furthermore for a compact set $K \subset X$, we denote $\mathscr{K}_{K} (X) $, to be the set of functions $\varphi \in \mathscr{K} (X)$, such that $ \spt \varphi \subset K$. 
    
    The space of Radon measures over $X$ is denoted by$\mathscr{M}_{+} (X)$.
    For another locally compact Hausdorff space $Y$, a proper, continuous map $f \colon X \rightarrow Y$, and $\mu \in \mathscr{M}_{+} (X)$, we define $f_{\sharp} \mu \in \mathscr{M}_{+} (Y)$ by, 
    \begin{equation*}
        \textstyle f_{\sharp} \mu (A) = \mu (f^{-1} (A)), \quad A \subset Y \ \text{Borel}.
    \end{equation*}
\end{propparagraph}

\begin{propparagraph}[Hausdorff measure]\label{para: hausdorff measure}
If $X$ is a metric space
with metric $d$, $\mathbf U_{d}(x,r)$ and $\mathbf B_{d}(x,r)$ respectively denote the open ball and closed balls, centered at $x\in X$, with radius $0<r<\infty$.
For $0 \leq s <  \infty$, we define the $s$ dimensional Hausdorff measure on $X$, with respect to metric $d$ by, $\mathscr{H}_{d}^{s} (A) = \lim_{\delta \downarrow 0} \mathscr{H}_{d, \delta}^{s} (A)$, for $A \subset X$, where
    \begin{eqnarray*}
         \mathscr{H}_{d, \delta}^{s} (A) = \inf \left\{ \sum_{i = 1}^{\infty} \left( \frac{\mathrm{diam}_{d} C_{i}}{2} \right)^{s} \, \colon \, A \subset \bigcup_{i = 1}^{\infty} C_{i} \subset X, \, \mathrm{diam}_{d} C_{i} \leq \delta \right\}.
    \end{eqnarray*}
We will often drop notational dependence on the metric $d$ when there is no ambiguity.
When $X = U \subset \mathbf{R}^{n}$, with the standard Euclidean metric, we opt to define the $s$ dimensional Hausdorff measure, $\mathscr{H}^{s} = \lim_{\delta \downarrow 0} \mathscr{H}^{s}_{\delta}$, with the standard renormalisation, 
\begin{equation*}
    \mathscr{H}_{\delta}^{s} (A) = \inf \left\{ \sum_{i = 1}^{\infty} \boldsymbol{\alpha} (s)\left( \frac{\mathrm{diam} C_{i}}{2} \right)^{s} \, \colon \, A \subset \bigcup_{i = 1}^{\infty} C_{i} \subset X, \, \mathrm{diam} C_{i} \leq \delta \right\},
\end{equation*}
where $\boldsymbol{\alpha} (s) = \pi^{- s / 2} \left( \int_{0}^{\infty} \exp (-r) r^{s / 2} \, d \mathscr{L}^{1}_{r} \right)^{-1}$.
\end{propparagraph}

\begin{propparagraph}[Densities]\label{para: densities}
    For a metric space $X$ with metric $d$, Borel measure $\mu$ on $X$, and $x \in X$, we denote,
    \begin{equation*}
        \Theta_{* d}^{m} (\mu ,x) = \liminf_{r \downarrow 0} \frac{\mu (\mathbf{B}_{d} (x, r))}{r^{m}}, \quad \Theta_{d}^{* m} (\mu ,x) = \limsup_{r \downarrow 0} \frac{\mu (\mathbf{B}_{d} (x, r))}{r^{m}},
    \end{equation*}
    When, $\Theta_{* d}^{m} (\mu ,x) = \Theta_{d}^{* m} (\mu ,x)$, we denote this common value by $\Theta_{d}^{m} (\mu, x)$.
    In the case of $X = U \subset \mathbf{R}^{n}$, with the standard Euclidean metric, we opt to use the upper and lower densities with the standard renormalisation,
    \[{\textstyle  \Density_{*}^{m} }(\mu ,x) = \liminf\limits_{r \downarrow 0} \frac{\mu (\mathbf{B} (x, r))}{\bfalpha(m)r^{m}}, \quad {\textstyle \Density^{* m}} (\mu ,x) = \limsup\limits_{r \downarrow 0} \frac{\mu (\mathbf{B} (x, r))}{\bfalpha(m)r^{m}}.\]
    When these upper and lower densities agree, $\Density^m(\mu,x)$ denotes their common value.
\end{propparagraph}

\begin{propparagraph}[Varifolds]\label{proppara: varifold notation}
As in \cite[2.3]{Allard1972OnTF}, we denote the set of $k$ dimensional linear subspaces of $\mathbf R^n$ by $\G(n,k)$, and we denote $\G_k(U) = U\times \G(n,k)$.
Associating $S \in \mathbf{G} (n, k)$ with the projection matrix onto $S$, the operator norm, which we denote by $\| \, \cdot \, \|$, yields a metric on $\mathbf{G} (n, k)$.
Again, following \cite[3.1, 3.5]{Allard1972OnTF}, we denote the space of varifolds and integral varifolds in $U$, by $\Var_k(U)$ and $\IVar_k(U)$ respectively.
For $\phi \colon U \rightarrow \mathbf{R}$, a non-negative, class 2 function with compact support, we denote $U_{\phi} = U \cap \{x  \colon \phi (x) > 0\}$, and for $V \in \mathbf{G}_{k} (U)$, we denote $V_{\phi}$ to be the restriction of $V$ to subsets of $\mathbf{G}_{k} (U_{\phi})$.
The generalized mean curvature of $V\in \Var_k(U)$ at $x\in U$, is defined in \cite[4.3]{Allard1972OnTF}, and denoted by $\bfh(V,x)$, whenever it exists. 
For a $k$ dimensional $C^{1}$ submanifold $M \subset U$, such that $\mathscr{H}^{k} (M \cap K) < + \infty$ for all compact $K \subset U$, we denote the integral varifold, 
\begin{equation*}
    \mathbf{v} (M) (\psi) = \int_{M} \psi (x, \mathrm{Tan} \, M) \, d \mathscr{H}^{k}_{x}, \quad \psi \in \mathscr{K} (\mathbf{G}_{k} (U)).
\end{equation*}
For $r > 0$, and $a \in \mathbf{R}^{n}$, we denote $\boldsymbol{\mu}_{a, r} (x) = r (x - a)$, and then for $V \in \mathbf{V}_{k} (U)$, we denote 
\begin{equation*}
    V^{a; r} (A) = r^{-k} V (\{ (x, S) \, \colon \, (\boldsymbol{\mu}_{a, r^{-1}} (x), S) \in A\},
\end{equation*}
for $A \subset \mathbf{G}_{k} (\boldsymbol{\mu}_{a, r^{-1}} (U))$ Borel, and note that, 
\begin{equation*}
    \textstyle \| V^{a; r} \| = r^{-k} (\boldsymbol{\mu}_{a, r^{-1}})_{\sharp} \| V \|
\end{equation*}
For $V \in \mathbf{V}_{k} (U)$, and $a \in U$, we define $\mathrm{Tan}^{k} (V, a) \in \mathbf{G} (n, k)$, when it exists, to be such that
    \begin{equation*}
        \textstyle \lim_{r \downarrow 0} V^{a; r} = \Density^{k} (\| V \|, a) \cdot \mathbf{v} (\mathrm{Tan}^{k} (V, x)).
    \end{equation*}
Note that this definition differs from that in \cite[2.8(3)]{Allard1972OnTF}, but agrees with the notion of tangent plane as given in \cite[38.1]{Simon-GMT}.
\end{propparagraph}

\begin{propparagraph}[Time slices of Radon measures over $J \times \mathbf{G}_{k} (U)$]
Let $V$ be a Radon measure on $J \times \mathbf{G}_{k} (U)$, then, similarly to \cite[3.4]{Menne-WDFV}, for $\mathscr{L}^{1}$ almost all $t \in J$, we can define the following Radon measures over $\mathbf{G}_{k} (U)$,
    \begin{equation*}
    \begin{split}
       \textstyle  \langle V, t \rangle (\alpha) &= \lim\limits_{\varepsilon \downarrow 0} \frac{1}{2 \varepsilon} \int \, \mathbf{1}_{\{ r \colon t - \varepsilon \leq r \leq t + \varepsilon \}} (s) \alpha (x, S) \, d V (s, x, S), \\ 
        \textstyle\langle V^{+}, t \rangle (\alpha) &= \lim\limits_{\varepsilon \downarrow 0} \frac{1}{\varepsilon} \int \mathbf{1}_{\{ r \colon t \leq r \leq t + \varepsilon \}} (s)  \, \alpha (x, S) \, d V (s, x, S), \\ 
        \textstyle\langle V^{-}, t \rangle (\alpha) &= \lim\limits_{\varepsilon \downarrow 0} \frac{1}{\varepsilon} \int \mathbf{1}_{\{ r \colon t - \varepsilon \leq r \leq t \}} (s) \, \alpha (x, S) \, d V (s, x, S),
    \end{split}
    \end{equation*}
for $\alpha \in \mathscr{K} (\mathbf{G}_{k} (U))$.
We analogously define, for $\mathscr{L}^{1}$ almost all $t \in J$, the Radon measures, $\langle \|V\|, t \rangle$, and $\langle \| V \|^{\pm}, t \rangle$, over $U$, for the weight measure $\|V\|$ over $J \times U$.
Note that when defined, $\| \langle V, t \rangle \| = \langle \| V \|, t \rangle$, and $\| \langle V^{\pm}, t \rangle \| = \langle \| V \|^{\pm}, t \rangle$.
\end{propparagraph}

\begin{propparagraph}[Parabolic metrics and their Hausdorff measures]\label{para: parabolic metrics}
    For $(t, x), \, (s, y) \in \mathbf{R} \times \mathbf{R}^{n}$ we denote the following parabolic metrics, 
\begin{eqnarray*}
    d ((t, x), (s, y)) &=& \left( |t - s| + |x - y|^{2} \right)^{1 / 2}, \\ 
    \rho ((t, x), (s, y)) &=& \max \left\{ |t - s|^{1 / 2}, |x - y| \right\}.
\end{eqnarray*}
and note that $\rho$ and $d$ are bi-Lipschitz equivalent.
We will also make use of the following notation for the parabolic `norm' associated to the metric $d$, 
\begin{equation*}
    \| (t, x) \| = \left( |t| + |x|^{2} \right)^{1 / 2}.
\end{equation*}
For every $(t_0,x_0) \in \mathbf R \times \mathbf R^n$ and $0<r<\infty$, we denote 
\begin{equation*}
    \mathbf P(t_0,x_0,r) = \{(t,x) : t_0-r^2 \leq t \leq t_0, \ |x-x_0|\leq r \}.
\end{equation*}
For $T \in \mathbf{G} (n, k)$, we denote $T^{\perp} \in \mathbf{G}(n, n - k)$ to be the orthogonal complement of $T$. 
Then letting $\mathbf{p}_{T}$ and $\mathbf{q}_{T}$ denote the orthogonal projections onto $\mathbf{R} \times T$ and $\{0\} \times T^{\perp}$ respectively, we have that
\[ \|\xi\|^2 = \| \bfp_{T} (\xi)\|^2 + \|\bfq_{T}(\xi)\|^2 \quad \text{for $\xi \in \mathbf R \times \mathbf R^n$}.\]
For  a positive integer $k$, by the uniqueness of uniformly distributed measures on separable metric spaces \cite[3.4]{Mattila-RectifiabilityBook}, we have positive finite constants $\boldsymbol{\omega} (k)$, and $\boldsymbol{\nu} (k)$, such that on $\mathbf{R} \times \mathbf{R}^{k}$,
\begin{equation*}
    \mathscr{H}_{d}^{k + 2} = \boldsymbol{\beta} (k)^{-1} \boldsymbol{\omega} (k) \mathscr{L}^{k + 1},
\end{equation*}
where $\boldsymbol{\beta} (k) = \mathscr{L}^{k + 1} (\mathbf{B}_{d} (0, 1))$, and, 
\begin{equation*}
    \mathscr{H}_{\rho}^{k + 2} = (2 \boldsymbol{\alpha} (k))^{-1}\boldsymbol{\nu} (k) \mathscr{L}^{k + 1}.
\end{equation*}
For $d$ and $\rho$ on $\mathbf{R} \times \mathbf{R}^{n}$, there exists a positive, finite constant $\gamma (n)$, such that for all $0 \leq s < + \infty$, 
\begin{equation*}
    \gamma (n)^{-s} \mathscr{H}_{d}^{s} \leq \mathscr{H}_{\rho}^{s} \leq \gamma (n)^{s} \mathscr{H}_{d}^{s}.
\end{equation*}
Thus, throughout the paper, for statements made in terms of either $\mathscr{H}_{d}^{s}$, or $\mathscr{H}_{\rho}^{s}$, an analogous statement in terms of the other measure, $\mathscr{H}_{\rho}^{s}$, or $\mathscr{H}_{d}^{s}$, often holds.
\end{propparagraph}

\subsection{Parabolic Rectifiability}\label{subsec: parabolic rectifiability}

\begin{definition}[\protect{\cite[3.1]{mattila2021parabolicrectifiabilitytangentplanes}}]\label{def: parabolic Lipschitz graph}
    For $0 < L < + \infty$, we say that a set $G \subset \mathbf{R} \times \mathbf{R}^{n}$ is a (vertical) parabolic ($k+2$, $L$)-Lipschitz graph if there exists a $T \in \mathbf{G} (n, k)$, $A \subset \mathbf{R} \times T$, and $g \colon A \rightarrow T^{\perp}$, such that, 
    \begin{equation*}
        |g (t, x) - g (s, y)| \leq L \| (t, x) - (s, y) \|,
    \end{equation*}
    and,
    \begin{equation*}
        G = \left\{ (t, x + g (t, x)) \, \colon \, (t, x) \in A \right\}.
    \end{equation*}
\end{definition}

\begin{definition}[\protect{\cite[1.2, 3.6]{mattila2021parabolicrectifiabilitytangentplanes}}]\label{def: parabolic rectifiability}
    A set $E \subset \mathbf{R} \times \mathbf{R}^{n}$ is a (vertical) parabolic $(k+2)$-rectifiable set, if for every $0 < L < + \infty$, there are ($k + 2$, $L$)-Lipschitz graphs $G_{1}, \, G_{2}, \, \ldots$, such that, 
    \begin{equation*}
        \mathscr{H}_{d}^{k + 2} \left( E \sim \bigcup_{i = 1}^{\infty} G_{i} \right) = 0.
    \end{equation*}
\end{definition}

\begin{definition}[\protect{\cite[4.4]{mattila2021parabolicrectifiabilitytangentplanes}}]
    For a Radon measure $\mu$ on $J \times U$, a non-zero Radon measure $\nu$ on $\mathbf{R} \times \mathbf{R}^{n}$ is said to be a tangent measure of $\mu$, at $a \in \mathbf{R} \times \mathbf{R}^{n}$, if there exists sequences of positive real numbers $c_{1}, \, c_{2}, \, \ldots$, and $r_{1}, \, r_{2}, \, \ldots$, with $r_{i} \rightarrow 0$, such that
    \begin{equation*}
        c_{i} (T_{a, r_{i}})_{\sharp}\mu \rightarrow \nu,
    \end{equation*}
    as Radon measures, where $T_{(t, x), r} (s, y) = (r^{-2} (s - t), r^{-1} (y - x))$.
    We denote the set of such Radon measures by $\mathrm{Tan} (\mu, a)$.
\end{definition}

\begin{lemma}[\protect{\cite[4.6(2)]{mattila2021parabolicrectifiabilitytangentplanes}}]\label{lem: equality of tangent measure for a.c. measures}
    For Radon measures $\nu$ and $\mu$ on $J \times U$, we have that if $\nu$ is absolutely continuous with respect to $\mu$, then $\mathrm{Tan} (\nu, a) = \mathrm{Tan} (\mu, a)$ for $\nu$ almost all $a \in \mathbf{R} \times \mathbf{R}^{n}$.
\end{lemma}

\begin{theorem}[\protect{\cite[4.12]{mattila2021parabolicrectifiabilitytangentplanes}}]\label{thm: rectifiability thm from Mattila}
    Let $E \subset \mathbf{R} \times \mathbf{R}^{n}$,  be $\mathscr{H}^{k + 2}_{d}$ measurable, and $\mathscr{H}_{d}^{k + 2} (E) < + \infty$. 
    Then the following are equivalent: 
    \begin{enumerate}
        \item $E$ is a (vertical) parabolic ($k+2$) rectifiable set.
        \item For $\mathscr{H}^{k+2}_{d}$ almost all $a \in E$, there exists a unique $T \in \mathbf{G} (n, k)$ such that, 
        \begin{equation*}
            \mathrm{Tan} (\mathscr{H}^{k + 2}_{d} \weight E, a) = \left\{ \lambda \mathscr{H}_{d}^{k + 2} \weight \mathbf{R} \times T \, \colon 0 < \lambda < + \infty \right\}. 
        \end{equation*}
    \end{enumerate}
\end{theorem}

\begin{theorem}[\protect{\cite[4.9]{mattila2021parabolicrectifiabilitytangentplanes}}]\label{thm: mattila existence of density a.e. for parabolic rectifiable sets}
    If $E \subset J \times U$ is $\mathscr{H}^{k + 2}_{d}$ (equivalently $\mathscr{H}_{\rho}^{k + 2}$) measurable and a (vertical) parabolic $(k+2)$ rectifiable set, then for $\mathscr{H}^{k + 2}_{d}$ (equivalently $\mathscr{H}_{\rho}^{k + 2}$) almost all $(t, x) \in E$, 
    \begin{equation*}
        \Theta^{k + 2}_{d} (\mathscr{H}^{k + 2}_{d} \weight E, t, x) = \boldsymbol{\omega} (k), \, \text{ and } \, \Theta^{k + 2}_{\rho} (\mathscr{H}^{k + 2}_{\rho} \weight E, t, x) = \boldsymbol{\nu} (k).
    \end{equation*}
\end{theorem}

\begin{propparagraph}\label{proppara: parabolic Besicovitch}
    A Besicovitch covering theorem has been proven for the parabolic metric $\rho$ on $\mathbf{R} \times \mathbf{R}^{n}$ (\cite[1.1]{Itoh-ParabolicBesicovitch}), and hence for any open set $\Omega \subset \mathbf{R} \times \mathbf{R}^{n}$, and $\mu \in \mathscr{M}_{+} (\Omega)$, the covering relation, 
    \begin{equation*}
        C = \{ ((t, x), \mathbf{B}_{\rho} ((t, x), r)) \, \colon \, 0 < r < \infty, \, \mathbf{B}_{\rho} ((t, x), r) \subset \Omega \}, 
    \end{equation*}
    is a $\mu$ Vitali relation (see \cite[2.9.16]{MR41:1976}).
    Indeed, let $K_{1}, \ K_{2}, \ldots$ be a compact exhaustion of $\Omega$, and consider any subset $A \subset \Omega$, and subcollection $D \subset C$, which is fine at each $a \in A$, that is
    \begin{equation*}
        \textstyle \inf \{ r \, \colon \, (a, \mathbf{B}_{\rho} (a, r)) \in D, \, a \in A \} = 0.
    \end{equation*}
    Now consider the following collection of balls, 
    \begin{equation*}
    \begin{split}
        \textstyle \mathscr{F} = \{ \mathbf{B}_{\rho} (a, r) \, \colon \, (a, \mathbf{B}_{\rho} (a, r)) & \in D, \ a \in A \cap K_{i} \cap \Omega, \\
        & \ \mathbf{B}_{\rho} (a, r) \subset \Omega, \ 0 < r \leq 1\}.
    \end{split}
    \end{equation*}
    Then, by \cite[1.1]{Itoh-ParabolicBesicovitch}, there exists a finite, positive integer $N = N (n)$, and a disjoint, countable subcollection $\mathscr{G} = \{ \mathbf{B}_{\rho} (a_{i}, r_{i}) \, \colon \, i \in \mathscr{P} \} \subset \mathscr{F}$, such that
    \begin{equation*}
        \textstyle \mu \left( A \cap K_{i} \cap \Lambda \sim \bigcup_{i = 1}^{\infty} \mathbf{B}_{\rho} (a_{i}, r_{i}) \right) \leq \frac{N}{N - 1} \mu (A \cap K_{i} \cap \Omega).
    \end{equation*}
    Thus by \cite[2.8.2]{MR41:1976}, the collection of balls, 
    \begin{equation*}
        \textstyle D (A) = \{ \mathbf{B}_{\rho} (a, r) \, \colon \, (a, \mathbf{B}_{\rho} (a, r)) \in D, \ a \in A \},
    \end{equation*}
    is $\mu$ adequate for $A$ (see \cite[2.8.1]{MR41:1976}), and thus $C$ is a $\mu$ Vitali relation for $\Omega$.
    
    The fact $C$ is a $\mu$ Vitali relation now allows us to directly apply the relevant theory in \cite[2.9]{MR41:1976} to obtain the following:
    \begin{enumerate}[label=\ref{proppara: parabolic Besicovitch}.\arabic{enumi}]
    \item\label{propara: differentiation of Radon measures} For Radon measures $\mu$ and $\nu$ on $\Omega$, with $\nu$ absolutely continuous with respect to $\mu$, we have (see \cite[2.9.1--7]{MR41:1976}) for all Borel sets $B \subset \Omega$ 
    \begin{equation*}
        \nu (B) = \int_{B} \mathbf{D} (\nu, \mu, C, a) \, d \mu_{a},  
    \end{equation*}
    where $\mathbf{D} (\nu, \mu, C, a) = \lim_{r \downarrow 0} \nu (\mathbf{B}_{\rho} (a, r)) / \mu (\mathbf{B}_{\rho} (a, r))$.
    \item\label{proppara: approximate continuity} For Radon measure $\mu$ on $\Omega$, and function $f$ mapping $\mu$ almost every point of $\Omega$ into a separable metric space $Y$, we have that (see \cite[2.9.13]{MR41:1976}) $f$ is $\mu$ measurable if and only if $f$ is $(\mu, C)$ approximately continuous at $\mu$ almost all points of $\Omega$, i.e. for $\mu$ almost all $a \in \Omega$, and all $0 < \varepsilon < \infty$, 
    \begin{equation*}
        \lim_{r \downarrow 0} \mu (\mathbf{B}_{\rho} (a, r) \cap \{ b \, \colon \, \sigma (f (a), f (b)) \geq \varepsilon \}) / \mu (\mathbf{B}_{\rho} (a, r)) = 0,
    \end{equation*}
    where $\sigma$ is the metric on $Y$.
    \end{enumerate}
\end{propparagraph}

\subsection{Brakke Flow}\label{subsec: Brakke Flow preliminaries}

\begin{definition}\label{Ilmanen-def}
(\cite[6.3]{Ilmanen-EllipticRegularization})
For $U$, an open subset of $\mathbf R^n$, $\mu$ a Radon measure over $U$, and $\phi : U \to \mathbf R$ is nonnegative, class $2$ function with compact support, if the following four conditions are satisfied 
    \begin{enumerate}[label=B-\arabic{enumi}]
        \item \label{integral}
         $\mu \weight U_\phi$ is a weight of some $V\in \IVar_k(U)$.
        \item \label{locally Radon}
        $\|\delta V_\phi\|$ is Radon measure.
        \item 
        \label{locally absolutely continuous}
     $\|\delta V_\phi\|$ absolutely continuous with respect to $\|V_\phi\|$.
        \item 
        \label{weighted L2 mean curvature} 
        $\int \phi(x)\cdot |\bfh(V_\phi,x)|^2 \ d \|V_\phi\|_x<\infty$ ,
    \end{enumerate}
    then we define 
    \begin{equation}
        \label{regular case}\tag{{$\mathscr B$}}
        \textstyle \mathscr B(\mu,\phi)  = \int -\phi (x)\cdot |\bfh(V_\phi,x)|^2 +\grad \phi(x)\bullet \bfh(V_\phi,x)\ d \|V_\phi\|_x
    \end{equation} 
      Otherwise, we define
    $\mathscr B(\mu,\phi)=-\infty$.
\end{definition}

\begin{definition}\label{def: Ilmanen def Brakke flow}
    A function $\mu \colon J \rightarrow \mathscr{M}_{+} (U)$ is a Brakke flow, if and only if, for every non-negative function $\phi \colon U \rightarrow \mathbb{R}$ class 2, with compact support, 
    \begin{equation*}
        \overline{\D}_{s} \mu (s) (\varphi)|_{t} \leq \mathscr{B} (\mu (t), \phi) \quad \text{for every $t\in J$}.
    \end{equation*}
\end{definition}

\begin{propparagraph}[\protect{\cite[4.1]{LW-Space-Time-GrassmannMeasureBrakkeFlow}}]\label{proppara: space-time-Grassmann measure from Brakke flow}
    Recall that for a Brakke flow $\mu \colon J \rightarrow \mathscr{M}_{+} (U)$, for $\mathscr{L}^{1}$ almost all $t \in J$, there exists a unique $V (t) \in \mathbf{IV}_{k} (U)$, such that $\| V (t) \| = \mu (t)$ (see \cite[3.10.6]{LW-Space-Time-GrassmannMeasureBrakkeFlow}). 
    Then, by \cite[3.10.8]{LW-Space-Time-GrassmannMeasureBrakkeFlow}, there exists a canonical Radon measure $V$ over $J \times \mathbf{G}_{k} (U)$, associated to $\mu$, given by, 
    \begin{equation*}
      \textstyle  V (\psi) = \int \int \psi (t, x, S) \, d V (t)_{(x, S)} \, d \mathscr{L}^{1}_{t}, \hspace{0.5cm} \psi \in \mathscr{K} (J \times \mathbf{G}_{k} (U)),
    \end{equation*}
    with an associated Radon measure over $J \times U$, 
    \begin{equation*}
        \textstyle \| V \| (f) = \int f (t, x) \, d V_{(t, x, S)} = \int \int f (t, x) \, d \mu (t)_{x} \, d \mathscr{L}^{1}_{t}, \hspace{0.5cm} f \in \mathscr{K} (J \times U).
    \end{equation*}
\end{propparagraph}

\begin{theorem}[\protect{\cite[Theorem B]{LW-Space-Time-GrassmannMeasureBrakkeFlow}}]\label{thm:Space time Brakke flow}
    Suppose $V$ is a Radon measure over $J\times \G_k(U)$, and that for every non-negative $\omega \in \mathscr D(J)$, and every non-negative, class function $2$ $\phi :  U \to \mathbf R$, with compact support, we have that
    \begin{equation}\label{eq:Brakke's distributional inequality}
        - \textstyle \int \omega'(t) \phi(x) \ d \|V\|_{(t,x)} 
        \leq \textstyle \int \omega(t) \mathscr B(\|\langle V,t\rangle \|,\phi) \ d \mathscr L^1_t \in \mathbf R. 
    \end{equation}
    Then the following six statements hold:
    \begin{enumerate}[label=\ref{thm:Space time Brakke flow}.\arabic{enumi}]
        \item \label{Space time Brakke flow:0}
        For every $\phi : U \to \mathbf R$ of class $2$ with compact support, 
        the function $\|\langle V,t\rangle \|(\phi)$ for $t\in J$ is of locally bounded variation.
        \item For every $\psi \in \mathscr{K} (\mathbf{G}_{k} (U))$, and $\omega \in \mathscr D(J)$, 
        \[ V_{(t,x, S)}(\omega(t)\cdot \psi(x, S)) = \textstyle \int \omega(t)\cdot \langle V,t\rangle (\psi) \ d \mathscr L^1_t,\]
        
        \item \label{Space time Brakke flow:1} $\langle \|V\|^-,\cdot \rangle$ is left continuous, with domain $J$, and 
        \[  \lim\limits_{t \downarrow a} \langle \|V\|^{-},t\rangle  \leq \langle \| V\|^{-},a\rangle \ \text{for $a\in J$}.\]
        \item \label{Space time Brakke flow:2}
        $\langle \| V \|^{+},\cdot \rangle$ is right continuous, with domain $J$, and 
        \[  \lim\limits_{t \uparrow a} \langle \|V\|^{+},t\rangle  \geq \langle \| V\|^{+},a\rangle  \ \text{for $a\in J$}.\]
        \item \label{Space time Brakke flow:3}
        $\langle \|V \|^+,\cdot \rangle $ and $\langle \| V \|^-,\cdot \rangle$ are  Brakke flows, and $\langle \|V \|^{-},\cdot \rangle \geq \langle \| V \|^{+},\cdot \rangle$. 
        \item \label{Space time Brakke flow:4} 
        For a function $\nu \colon J \rightarrow \mathscr{M}_{+} (U)$, the following are equivalent:
        \begin{enumerate}
            \item $\nu$ is a Brakke flow, with 
            \[  \textstyle  \|V\|(\theta) = \int \int \theta(t,x) \ d \nu(t)_x \ d \mathscr L^1_t, \ \text{for $\theta \in \mathscr K(J\times U)$},\]
            \item for every $t \in J$, 
            \[  \langle \| V \|^{-} ,t \rangle \geq \nu(t) \geq \langle \| V \|^{+},t\rangle .\]
        \end{enumerate}
    \end{enumerate}
\end{theorem}

\begin{definition}\label{def: space-time-Grassmann Brakke flow}
    We say a Radon measure $V$ over $J \times \mathbf{G}_{k} (U)$ is a \textit{space-time-Grassmann Brakke flow} if and only if the  following three statements hold,
    \begin{enumerate}
        \item For $\mathscr{L}^{1}$ almost all $t \in J$, $\langle V, t \rangle \in \mathbf{IV}_{k} (U)$.
        \item For every non-negative, class 2 function $\phi \colon U \rightarrow \mathbf{R}$, with compact support $\mathscr B(\|\langle V,\cdot\rangle \|,\phi) \in \mathbf{L}_{1}^{\text{loc}} (J)$.
        \item For every non-negative $\omega \in \mathscr{D} (J)$, and every non-negative, class 2 function $\phi \colon U \rightarrow \mathbf{R}$, with compact support, we have that, 
    \begin{equation*}
        - \textstyle \int \omega'(t) \phi(x) \ d \|V\|_{(t,x)} 
        \leq \textstyle \int \omega(t) \mathscr B(\|\langle V,t\rangle \|,\phi) \ d \mathscr L^1_t
    \end{equation*}
    for every $\omega \in \mathscr D(J)$.
    \end{enumerate}
\end{definition}

\begin{remark}
    From \cite[4.12]{LW-Space-Time-GrassmannMeasureBrakkeFlow}, we have that Definition \ref{def: Ilmanen def Brakke flow} and Definition \ref{def: space-time-Grassmann Brakke flow} are equivalent.
\end{remark}

\begin{propparagraph}\label{proppara: extended Brakke inequality}
    (\cite[4.13]{LW-Space-Time-GrassmannMeasureBrakkeFlow})
    Suppose $V$ is a space-time-Grassmann Brakke flow over $J \times \mathbf{G}_{k} (U)$, $\phi \colon J \times U \rightarrow \mathbf{R}$ is a non-negative class 2 function with compact support, then the following inequality holds, 
    \begin{eqnarray*}
        - \textstyle \int \partial_{s} \phi (s, x)|_{t} \, d \| V \|_{(t, x)} &\leq& \textstyle\int - \phi(t, x) |\mathbf{h} (\langle V, t \rangle, x)|^{2} \\ 
        && \hspace{1cm} + \mathrm{grad}_{y} \phi (t, y)|_{x} \bullet \mathbf{h} (\langle V, t \rangle, x) \, d \| V \|_{(t, x)}.
    \end{eqnarray*}
\end{propparagraph}

\begin{propparagraph}\label{proppara: def of representative Brakke Flow}
    For a space-time-Grassmann Brakke flow $V \in \mathscr{M}_{+} (J \times \mathbf{G}_{k} (U))$, we say that any Brakke flow $\nu \colon J \rightarrow \mathscr{M}_{+} (U)$ is a \textit{representative} of $V$, if either of the two equivalent conditions of \ref{Space time Brakke flow:4} hold for $\nu$.
\end{propparagraph}

\begin{remark}
    By \ref{Space time Brakke flow:4}, a space-time-Grassmann Brakke flow $V$ is uniquely determined by its space-time weight $\| V \|$.
\end{remark}

\begin{proposition}\label{proppara: L1 to L infinity mass bound}
    Let $I \subset J$ be a non-degenerate compact interval, and $K \subset U$ be compact. 
    Then, there exists a positive, finite constant $C = C (I, K)$, such that for any space-time-Grassmann Brakke flow $V$ on $J \times U$, 
    \begin{equation*}
        \sup \{ \langle \| V \|^{-}, t \rangle (K) \colon t \in I \} \leq C (I, K) \| V \| (J \times U).
    \end{equation*}
\end{proposition}

 \begin{proof}
     First we recall, from \cite[3.10.1]{LW-Space-Time-GrassmannMeasureBrakkeFlow} (c.f. \cite[3.3]{Lahiri2017EqualityOT}), that for all $x \in K$, $T \in J$ and $r > 0$ such that $\mathbf{B} (x, 2 r) \subset U$, we have that, 
    \begin{eqnarray*}
        && \sup \left\{ \langle \| V \|^{-}, t' \rangle (\mathbf{B} (x, r)) \, \colon t' \in J \cap \{t \colon T \leq t \leq T + r^{2} (2 k)^{-1}\} \right\} \\ 
        && \hspace{2cm} \leq 16 \langle \| V \|^{-}, T \rangle (\mathbf{B} (x, 2 r)).
    \end{eqnarray*}
    One can find a finite collection of balls $\{ \mathbf{B} (x_{i}, r_{i}) \colon i =1, \ldots, p\}$, such that for each $i = 1, \ldots, p$, $x_{i} \in K$, $\mathbf{B} (x_{i}, 2 r_{i}) \subset U$, and, 
    \begin{equation*}
        \textstyle K \subset \bigcup_{i = 1}^{p} \mathbf{B} (x_{i}, r_{i}).
    \end{equation*}
    We may also impose that $r_{1} \leq r_{2} \leq \ldots \leq r_{p}$.
    Now, let $a = \inf I$, $b = \sup I$, and $M = \| V \| (J \times U)$, and we may assume $M<\infty$. Taking some $s \in J \cap \{ t \colon \, a - r_{1}^{2} (4 k)^{-1} \leq t< a \}$, there exists an $s_{1} \in J \cap \{ t  \, \colon \, s \leq t \leq a \}$, such that, 
    \begin{equation*}
        \| \langle V^{-}, s_{1} \rangle \| (U) \leq \frac{M}{a - s}. 
    \end{equation*}
    Therefore, 
    \begin{equation*}
        \sup \left\{ \| \langle V^{-}, t' \rangle \| (K) \, \colon \, a \leq t' \leq a + r_{1}^{2}(4k)^{-1} \right\} \leq \frac{16 p M}{a - s},
    \end{equation*}
    and the conclusion follows by appropriately subdividing $\{ t : a \leq t \leq b\}$ into intervals of length at most $r_{1}^{2} (4 k)^{-1}$.
 \end{proof}

\begin{propparagraph}
    For a point $(t_{0}, x_{0}) \in \mathbf{R} \times \mathbf{R}^{n}$, we recall the $k$-dimensional backwards heat kernel, centred at $(t_{0}, x_{0})$,
    \begin{equation*}
        \Phi_{(t_{0}, x_{0})} (t, x) = (4 \pi (t_{0} - t))^{-\frac{k}{2}} \exp (- (4 (t_{0} - t))^{-1} |x - x_{0}|^{2}),
    \end{equation*}
    for points $(t, x) \in \mathbf{R} \times \mathbf{R}^{n}$, with $- \infty < t < t_{0}$.
\end{propparagraph}

\begin{theorem}[see \protect{\cite[4.13]{EckerMCFBook}}, \protect{\cite[6.2]{Lahiri2014}}, \protect{\cite[3.4]{tonegawa2019brakke}}]\label{prop:monotonicity formula}
    Suppose $V$ be a space-time-Grassmann Brakke flow over $J \times \mathbf{G}_{k} (U)$, $(t_{0}, x_{0}) \in \mathbf{R} \times \mathbf{R}^{n}$, and $\phi \colon J \times U$ a non-negative class 2 function, with compact support, then for all non-negative $\omega \in \mathscr{D} (J \cap \{ t \colon  t < t_{0} \})$, the following inequality holds
    \begin{eqnarray*}
        && \textstyle - \int \omega'(t) \Phi_{(t_{0}, x_{0})} (t, x) \phi (t, x) \, d \| V \|_{(t, x)} \\
        && \textstyle\hspace{0.5cm} \leq - \int \omega (t) \left| \mathbf{h} (\langle V, t \rangle, x) + \frac{S^{\perp}_{\natural} (x - x_{0})}{2 (t_{0} - t)} \right|^{2} \Phi_{(t_{0}, x_{0})} (t, x) \phi(t,x) \, d V_{(t, x, S)} \\
        &&\textstyle \hspace{1cm} + \int \omega (t) (\partial_{s} \phi (s, x )|_{t} - S_{\natural} \bullet \D^{2}_{y} \phi (t, y))|_{x} \Phi_{(t_{0}, x_{0})} (t, x) \, d V_{(t, x, S)}.
    \end{eqnarray*}
\end{theorem}

\begin{propparagraph}
    From \ref{prop:monotonicity formula}, one obtains the existence of the Gaussian density of $V$ at points $(t_{0}, x_{0}) \in J \times U$, denoted $\GDensity(V, t_{0}, x_{0})$ (see \cite[4.16(4)]{EckerMCFBook}).
    In particular (see \cite[(4.17)]{EckerMCFBook}), there exists a finite, non-negative number $\GDensity (V, t_{0}, x_{0})$, such that for any function $f \in \mathscr{K} (J \times U)$, and represetative $\mu$ of $V$,
    \begin{equation*}
        f (t_{0}, x_{0}) \GDensity (V, t_{0}, x_{0}) =\textstyle \lim\limits_{t \uparrow t_{0}} \int f (t, x) \rho_{(t_{0}, x_{0})} (t, x) \, d \mu(t)_{x}.
    \end{equation*}
\end{propparagraph}

\begin{proposition}[Properties of the Gaussian Density \protect{\cite[4.19, 4.20, 4.23]{EckerMCFBook}}]\label{prop:properties of Gaussian density}
    Suppose $V$ is a space-time-Grassmann Brakke flow over $J \times \mathbf{G}_{k} (U)$, then the following three statements hold: 
    \begin{enumerate}[label=\ref{prop:properties of Gaussian density}.\arabic{enumi}]
        \item \label{prop:properties of Gaussian density:1} If $(t_i,x_i) \to (t,x)$ as $i\to\infty$, then 
        \[ \limsup \limits_{i\to\infty} \GDensity(V,t_i,x_i) \leq \GDensity(V,t,x).\]
        \item \label{prop:properties of Gaussian density:2} For $\| V \|$ almost all $(t, x) \in J \times U$, we have that
        \begin{equation*}
            \Theta (V, t, x) \geq \textstyle \Density^{k} (\| \langle V, t \rangle \|, x).
        \end{equation*}
        \item \label{prop:properties of Gaussian density:3} We have that, 
        \begin{equation*}
            \spt  \| V \| =  J \times U \cap \{ (t, x) \colon \GDensity (V, t, x) \geq 1\} ,
        \end{equation*}
        and
        \begin{equation*}
            (J \times U) \sim \spt \| V \| = J \times U \cap \{ (t, x) \colon \GDensity (V, t, x) = 0 \}.
        \end{equation*}
    \end{enumerate}
\end{proposition}

\begin{proposition}[Mean value inequality \protect{\cite[4.25]{EckerMCFBook}}]\label{prop: Brakke flow mean value inequality}
    There exists a finite, positive constant $c$, dependent only $n$, and $k$, such that the following holds. 
    Suppose $V$ is a $k$-dimensional space-time-Grassmann Brakke flow over $J \times \mathbf{G}_{k} (U)$, and $f \colon J \times U$ is non-negative, class 2 function, such that for $V$ almost all $(t, x, S) \in J \times U \times\mathbf{G} (n, k)$, 
    \begin{equation*}
        \partial_{s} f (s, x)|_{t} - S_{\natural} \bullet \D_{y} \grad_{y} f (t, y)|_{x} \leq 0.
    \end{equation*}
    Then, for $(t_{0}, x_{0}) \in  \spt \| V \| $, and $r > 0$  such that $\mathbf{P} (t_{0}, x_{0}, r) \subset J \times U$, we have that, 
    \begin{equation*}
       \textstyle  c \cdot f (t_{0}, x_{0})^{2} \leq \frac{1}{r^{k + 2}} \int_{\mathbf{P} (t_{0}, x_{0}, r)} f(t, x)^{2} \, d \| V \|_{(t, x)}.
    \end{equation*}
\end{proposition}

\begin{proposition}\label{prop: density bounds}
Suppose $V$ is a $k$-dimensional space-time-Grassmann Brakke flow over $J \times \G_{k} (U)$, then the following holds: 
    \begin{enumerate}[label=\ref{prop: density bounds}.\arabic{enumi}]
        \item\label{prop: lower density bound} There exists a finite, positive constant $c$, dependent only on $n$ and $k$, such that for all $(t, x) \in \spt  \| V \|$,
        \begin{equation*}
            \Theta_{* \rho}^{k + 2} (\| V \|, t, x) \geq  c,
        \end{equation*}
        and thus $\mathscr{H}_{\rho}^{k + 2} \weight \mathrm{spt} \, \| V \| \leq c^{-1} \| V \|$.
        \item\label{prop: upper density bounds} Suppose $\Gamma$ is a non-negative, finite constant, such that $\| V \| (J \times U) \leq \Gamma <\infty$, and $J'$, and $U'$ are relatively compact subsets of $J$, and $U$, respectively, then there exists a non-negative, finite constant $C$, dependent only on $n$, $k$, $\Gamma$, $J'$, $J$, $U'$, and $U$, such that, 
        \begin{equation*}
            \Theta_{\rho}^{* k + 2} ( \| V \|, (t, x)) \leq C \ \quad \text{for $(t,x)\in J'\times U'$},
        \end{equation*}
        and thus $\| V \| \leq 2^{k + 2} C \mathscr{H}_{\rho}^{k + 2} \weight \mathrm{spt} \, \| V \|$, on $J' \times U'$.
    \end{enumerate}
\end{proposition}

\begin{proof}
    The lower density bound follows directly from \ref{prop: Brakke flow mean value inequality} with $f = 1$.
    For the upper density bound, consider $\mathbf{B}_{\rho} (t_{0}, x_{0}, 2r) \subset J \times U$, $0 < s < r$, and a class 2 function $\eta \colon \mathbf{R} \rightarrow \mathbf{R}$, which satisfies $\eta (t) = 0$, for $-\infty < t \leq 0$, $\eta (t) = 1$, for $1 \leq t < + \infty$, and $0 \leq \eta' (t) \leq 2$, and $|\eta'' (t)| \leq 8$, for all $t \in \mathbf{R}$. 
    We then define, 
    \begin{equation*}
        \phi (t, x) = \eta (r^{-2} (t - t_{0} + 2 r^{2})) \eta (r^{-2} (t_{0} + 2 r^{2} - t))\eta (r^{-1} (2 r - | x - x_{0}|)),
    \end{equation*}
    and, 
    \begin{equation*}
        \omega (t) = \begin{cases}
            s^{2}r^{-2} (t - (t_{0} - 4r^{2})), & t \in [t_{0} - 4r^{2}, t_{0} - 2 r^{2}], \\
            2 s^{2}, & t \in [t_{0} - 2 r^{2}, t_{0} - s^{2}], \\
            t_{0} + s^{2} - t, & t \in [t_{0} - s^{2}, t_{0} + s^{2}], \\ 
            0, & t \leq t_{0} - 4r^{2}, \, t \geq t_{0} + s^{2}.
        \end{cases}
    \end{equation*}
    The upper bound then follows from \ref{prop:monotonicity formula}, with $\phi (t, x)$ as above, an appropriate class 2 approximation of $\omega (t)$, and centreing the heat kernel at $(t_{0} + 2 s^{2}, x_{0})$.
    The further statements then follow from \cite[2.10.19]{MR41:1976}
\end{proof}

\begin{propparagraph}\label{proppara: a.c. of weight measure and hausdorff measure for Brakke flows}
    From \ref{prop: density bounds} $\mathscr{H}^{k + 2}_{\rho} \weight \mathrm{spt} \| V \|$, (and equivalently $\mathscr{H}^{k + 2}_{d} \weight \mathrm{spt} \| V \|$) are Radon measures over $J \times U$.
    Moreover, $\| V \|$ and $\mathscr{H}_{\rho}^{k + 2} \weight \mathrm{spt} \, \| V \|$ (and equivalently $\mathscr{H}^{k + 2}_{d} \weight \mathrm{spt} \| V \|$), are absolutely continuous with respect to eachother.
\end{propparagraph}

\section{Convergence of Brakke Flows}\label{sec: convergence of Brakke Flows}

\begin{definition}
    For a Radon measure $V$ over $J \times \mathbf{G}_{k} (U)$, we define
    \begin{equation*}
      \textstyle  \delta V (g) = \int S  \bullet \D_{z} g (t, z)|_{x} \  d V_{(t, x, S)} \ \text{for $g\in \mathscr D(J\times U,\mathbf R^n)$}.
    \end{equation*}
\end{definition}

\begin{proposition}\label{prop: first variation is a space-time Radon measure}
    Let $V$ be a space-time-Grassmann Brakke flow on $J \times \mathbf{G}_{k} (U)$, then $\| \delta V \|$ is a Radon measure over $J \times U$. 
\end{proposition}

\begin{proof}
    We may assume that $\| V \| ( J \times U) < + \infty$.
    For $\mathscr{L}^{1}$ almost all $t \in J$, 
    \begin{equation*}
        \textstyle \int S \bullet D_{y} g (t, y)|_{x} \, d \langle V, t \rangle_{(x, S)} = - \int g (t, x) \bullet \mathbf{h} (\langle V, t \rangle, x) \, d \| V \|_{x}.
    \end{equation*}
    and thus by \cite[4.1.2, Theorem C]{LW-Space-Time-GrassmannMeasureBrakkeFlow} we have that, 
    \begin{equation*}
        \textstyle \delta V (g) = - \int g (t, x) \bullet \mathbf{h} (\langle V ,t \rangle, x) \, d \| V \|_{(t, x)}.
    \end{equation*}
    Then for $I \subset J$, and $K \subset U$, both compact, by \ref{proppara: extended Brakke inequality}, with an appropriate choice of test function $\phi$, we obtain that there exists a finite, non-negative constant $C = C (I, K)$, such that
    \begin{equation*}
       \textstyle  \int_{I \times K} |\mathbf{h} (\langle V, t \rangle, x)|^{2} \, d \| V \|_{(t, x)} \leq C (I, K) \| V \| (J \times U).\qedhere
     \end{equation*}
\end{proof}

\begin{propparagraph}\label{proppar:rectifiablity of Brakke flow}
    Suppose $V$ is a space-time-Grassmann Brakke flow over $J \times U$, and $W$ is a Radon measure over $J \times \mathbf{G}_{k} (U)$, such that $\| V \| = \| W \|$, and $\| \delta W \|$ is a Radon measure over $J \times U$, then $V = W$. 
    Indeed, for $\alpha \in \mathscr{K} (\mathbf{G}_{k} (U))$, we define, 
    \begin{equation*}
        \textstyle R_{\alpha} (\omega) = \int \omega (t) \, \alpha (x, S) \, d W_{(t, x, S)} \ \text{for $\omega \in \mathscr D(J)$}.
    \end{equation*}
    Taking $\phi \in \mathscr{K} (U)$, such that $|\alpha (x, S)| \leq \phi (x)$ for all $x \in U$, we have that, 
    \begin{equation*}
        \textstyle |R_{\alpha} (\omega)| \leq R_{\phi} (\omega) = \int \omega (t) \| \langle V, t \rangle \| (\phi) \, d \mathscr{L}^{1}_{t},
    \end{equation*}
    which implies that $\| R_{\alpha} \|$ is absolutely continuous with respect to $\mathscr{L}^{1}$, and thus (see \cite[3.4(2)]{Menne-WDFV}),
    \begin{equation*}
        \textstyle \int \omega (t) \alpha (x, S) \, d W (t, x, S) = \int \omega (t) \langle W, t \rangle (\alpha) \, d \mathscr{L}^{1}_{t},  
    \end{equation*}
    with $\| \langle W, t \rangle \| = \| \langle V, t \rangle \|$ for $\mathscr{L}^{1}$ almost all $t \in J$.
    We therefore have that
    \begin{equation*}
        \textstyle \delta W_{(t, x)} (\omega (t) X (x)) = \int \omega (t) \delta \langle W, t\rangle (X) \, d \mathscr{L}^{1}_{t}.
    \end{equation*}
    which combined with the fact that $\| \delta W \|$ is a Radon measure over $J \times U$, implies (see \cite[3.4]{Menne-WDFV}) that for $\mathscr{L}^{1}$ almost all $t \in J$, $\| \delta \langle W, t \rangle \|$ is a Radon measure over $U$. 
    Thus, as $\langle V, t \rangle \in \mathbf{IV}_{k} (U)$ for $\mathscr{L}^{1}$ almost all $t$, by \cite[5.5(1)]{Allard1972OnTF}, for $\mathscr{L}^{1}$ almost all $t$, $\langle W, t \rangle = \langle V, t \rangle$.
\end{propparagraph}

\begin{lemma}\label{convergence of monotone BV}
    Consider $f,f_i : J \to \mathbf R$, and suppose that the $f_{i}$ are monotone functions, such that
    \[\textstyle  \lim\limits_{i\to\infty}\int |f-f_i| \ d \mathscr L^1 = 0.\]
    Then for every  $t$ at which $f$ is continuous, $f(t) = \lim\limits_{i\to\infty}f_i(t)$. 
\end{lemma}
\begin{proof}
    Taking $t \in J$ such that $f$ is continuous at $t$, then for each $0<\varepsilon<\infty$, there exists $0<\delta<\infty$ and 
    $n_0\in \mathscr P$ such that 
    \[ |f(t)-f(s)| < \varepsilon \quad \text{if $|t-s|<\delta$},\]
    \[ \mathscr L^1(J \cap \{ s : |f_i(s)-f(s)|>\varepsilon\}) < \delta \quad \text{for $i\geq n_0$}.\]
    For every positive integer $i\geq n_0$, we may
    choose $y_i,z_i$ such that 
    \[ t-\delta < y_i < t < z_i < t + \delta, \quad f_i(y_i),f_i(z_i) \in \mathbf B(f(t),2 \varepsilon),\]
    the monotonicity implies $f_i(t)\in \mathbf B(f(t),2\varepsilon)$.
\end{proof}

\begin{proposition}[Convergence of Brakke Flows]\label{prop: convergence of Brakke flows} 
    Let $V_{1}, \, V_{2}, \, \ldots$, and $V$ be space-time-Grassmann Brakke flows on $J \times U$, then the following three statements are equivalent: 
    \begin{enumerate}[label=\ref{prop: convergence of Brakke flows}.\arabic{enumi}]
         \item \label{prop: convergence of Brakke flows:1} $V = \lim\limits_{i\to\infty} V_i$. 
        \item \label{prop: convergence of Brakke flows:2} $\|V\| = \lim\limits_{i\to\infty} \|V_i\|$.
        \item \label{prop: convergence of Brakke flows:3} 
        There exists a countable set $D \subset J$ such that for every representative $\mu_i$ of $V_{i}$, and every $t\in J \sim D$,
        \[
            \| \langle V, t \rangle \| = \lim\limits_{i\to\infty} \mu_i(t).
        \]
    \end{enumerate}
\end{proposition}

\begin{proof}
    Clearly \ref{prop: convergence of Brakke flows:3} implies \ref{prop: convergence of Brakke flows:2}.
    We also have that \ref{prop: convergence of Brakke flows:2} implies \ref{prop: convergence of Brakke flows:1}. 
    Indeed, we first assume that $\sup_{i} \| V_{i} \| (J \times U) < + \infty$, and then for $\{V'_i\}$ a subsequence of $\{ V_i\}$, we may assume that there exists a Radon $W$ over $J \times \G_k(U)$ such that $V'_i \to W$ as $i\to\infty$.
    Moreover, as in the proof of \ref{prop: first variation is a space-time Radon measure}, for $I \subset J$, and $K \subset U$ compact, there exists a non-negative constant $C (I, K)$, such that, 
    \begin{equation*}
        \textstyle \sup_{i} \| \delta V_{i}'\| (I \times K) \leq C (I, K) \sup_{i} \| V_{i}' \| (J \times U) < + \infty.
    \end{equation*}
    Therefore $\| \delta W \|$ is a Radon measure over $J \times U$, and thus by \ref{proppar:rectifiablity of Brakke flow} $W = V$.
    
    Lastly we show that \ref{prop: convergence of Brakke flows:1} implies \ref{prop: convergence of Brakke flows:3}.
    Let $\mu_i$ and $\mu$ be representatives of $V_i$ and $V$ respectively, and $\phi : U \to \mathbf R$ a non-negative class $2$ function with compact support.
    We will prove the statement with the fixed function $\phi$, noting that the complete statement follows from taking an appropriate countable dense subset of $\mathscr{K}_{c} (U)$.
    We first show that there exists an at most countable subset $D$ of $J$ such that for all $t \in J \sim D$, 
    \begin{equation*}
        \textstyle \lim_{i \rightarrow \infty} \mu_{i} (t) (\phi) = \mu (t) (\phi).
    \end{equation*}
    Define, $g_i,g : J \to \mathbf R$ by 
    \[ g_i(t) = \mu_i(t)(\phi) \quad \text{for $t\in J$ and $i\in\mathscr P$},\]
    \[ g(t) = \mu(t)(\phi) \quad \text{for $t\in J$ }.\]
    After potentially considering a relatively compact subset of $J$, we may assume that $J$ is bounded, and by \ref{proppara: L1 to L infinity mass bound}, that $\sup_{i} \sup_{t \in J} \mu_{i} (t)(U_{\phi}) < + \infty$. 
    Thus by \cite[3.10.2]{LW-Space-Time-GrassmannMeasureBrakkeFlow}, and \cite[A.2]{Lahiri2017EqualityOT}, there exists a constant $0<L<\infty$, such that 
    \[ f_i(t) = g_i(t) - L \cdot t \quad \text{for $t\in J$ and $i\in\mathscr P$},\]
    \[ f(t) = g(t) - L \cdot t \quad \text{for $t\in J$ },\]
    are nonincreasing.
    This implies that there exists an at most countable set $D$, such that $f$ is continuous at points $t \in J \sim D$.
    The conclusion of our above claim then follows from \ref{convergence of monotone BV}, if we can show that, 
    \begin{equation}\label{eq:L_1 convergence}
        \textstyle \lim\limits_{i \rightarrow 0} \int_{I} |f_{i} (t) - f(t)| \, d \mathscr{L}^{1}_{t} = 0, \quad \text{for all compact $I \subset J$.}
    \end{equation}
    As $V_{i} \rightarrow V$, we have that for all $\omega \in \mathscr{D} (J)$,
    \begin{equation}\label{eq:weak convergece of f_i}
        \textstyle \lim\limits_{i \rightarrow \infty} \int \omega (t) f_{i} (t) \, d \mathscr{L}^{1}_{t} = \int \omega (t) f (t) \, d \mathscr{L}^{1}_{t}.
    \end{equation}
    Moreover, for $I$, a relatively compact subset of $J$, there exists a finite, non-negative constant $C (I)$, such that for non-negative $\omega \in \mathscr{D} (J)$, with $\mathrm{spt} \, \omega \subset I$,  we have that
    \begin{equation*}
        0 \leq \textstyle \int \omega' (t) f_i (t) d \mathscr L^1_t \leq C(I) \sup_{t} \omega (t)\int_{J} [\sup_{t} \mu_{i} (t)(\phi) + L t] \ d \mathscr L^1_t.
    \end{equation*}
    Thus $f_{1}, \, f_{2}, \ldots$ are uniformly bounded in the space of functions of locally bounded variation. 
    One then verifies (\ref{eq:L_1 convergence}) via a contradiction argument using (\ref{eq:weak convergece of f_i}) and \cite[3.23]{ambrosio2000functions}, hence verifying our initial claim.
    We finally note that from \ref{Space time Brakke flow:2}, \ref{Space time Brakke flow:3}, and \ref{Space time Brakke flow:4}, we may take
    \begin{equation*}
        D = \{ t \in J \, \colon \| \langle V^{-}, t \rangle \| (\phi) \not= \| \langle V^{+}, t \rangle \| (\phi) \},
    \end{equation*}
    and thus for $t \in J \sim D$, $\mu (t) (\phi) = \| \langle V, t \rangle \| (\phi)$, concluding that \ref{prop: convergence of Brakke flows:1} implies \ref{prop: convergence of Brakke flows:3}.
\end{proof}

\begin{definition}
    For space-time-Grassmann Brakke flows $V_{1}, \, V_{2}, \, \ldots$, and $V$, we say that $V_{1}, \, V_{2}, \, \ldots$ converge as Brakke flows to $V$, if any of the three equivalent statements from \ref{prop: convergence of Brakke flows} hold.
\end{definition}

\begin{theorem}[Compactness of Brakke Flows, \protect{\cite[7.1]{Ilmanen-EllipticRegularization}}, \protect{\cite[3.7]{tonegawa2019brakke}}]\label{thm:compactness of Brakke flows}
    Let $V_{1}, \, V_{2}, \, \ldots$ be a sequence of space-time-Grassmann Brakke flows such that for all compact $I \subset J$, and $K \subset U$, we have that, 
    \begin{equation*}
        \limsup_{i \rightarrow \infty} \| V_{i} \| (I \times K) < + \infty,
    \end{equation*}
    then there exists a subsequence $V_{i_{1}}, \, V_{i_{2}}, \, \ldots$, and a space-time-Grassmann Brakke flow $V$ over $J \times \mathbf{G}_{k} (U)$, such that $V_{i_{1}}, \, V_{i_{2}}, \, \ldots$ converge as Brakke flows to $V$.
\end{theorem}

\section{Tangent Flows and Parabolic Rectifiability}\label{sec: parabolic rectifiability}

\begin{propparagraph}[Flat Brakke Flows]\label{prop: flat Brakke flows}
    Let $V$ be a space-time-Grassmann Brakke flow on $J \times \mathbf{U} (0, R)$, for a connected open interval $J$, and suppose that there exists  $T \in \mathbf{G} (n, k)$, such that $\mathrm{Tan}^{k} ( \langle V, t \rangle, x) = T$, for $\| V \|$ almost all $(t, x) \in J \times \mathbf{U} (0, R)$, 
    and $\delta V = 0$.
    Then, there exists an at most countable set $S \subset T^{\perp} \cap \mathbf{U} (0, R)$, and functions, 
    \begin{equation*}
        \theta^{+}, \, \theta^{-} \colon J \times S \rightarrow \mathscr{P} \cup \{0\}, 
    \end{equation*}
    such that for each $x \in S$, 
    \begin{eqnarray*}
        && t \mapsto \theta^{+} (t, x), \, \text{is a non-increasing, right continuous function}, \\
        && t \mapsto \theta^{-} (t, x), \, \text{is a non-increasing, left continuous function},
    \end{eqnarray*}
    and, 
    \begin{equation*}
        \langle V^{+}, t \rangle = \sum_{x \in S} \theta^{+} (t, x) \cdot \mathbf{v} (T_{x} \cap \mathbf{U} (0, R)), \quad \langle V^{-}, t \rangle = \sum_{x \in S} \theta^{-} (t, x) \cdot \mathbf{v} (T_{x} \cap \mathbf{U} (0, R)),
    \end{equation*}
    where $T_{x} = \{ x + y \colon y \in T\}$.
    Moreover, we have that $\theta^{+} (t, x) \leq \theta^{-} (t, x)$, for all $(t, x) \in J \times S$, with equality for all, except for an at most countable $t \in J$. 
    Furthermore, for all $t \in J$, the sets
    \begin{equation*}
        S^{+} (t) = \{ x \in S \colon \theta^{+} (t, x) > 0\} \subset S^{-} (t) = \{ x \in S \colon \theta^{-} (t, x) > 0\},
    \end{equation*}
    are discrete, and relatively closed in $T^{\perp} \cap \mathbf{U} (0, R)$.
\end{propparagraph}

\begin{proof}
    For $\mathscr{L}^{1}$ almost all $t \in J$, by \cite[3.1]{Menne2010}, we have that there exists an at most countable, discrete, closed set $S (t) \subset T^{\perp}$, and a function $\theta (t) \colon S (t) \rightarrow \mathscr{P}$, such that, 
    \begin{equation*}
        \langle V ,t \rangle = \sum_{x \in S (t)} \mathbf{v} (T_{x}, \theta (t) (x)). 
    \end{equation*}
    Moreover, for all class $2$, non-negative functions $\phi \colon U \rightarrow \mathbf{R}$, with compact support, $\mathscr{B} (\| \langle V, t \rangle \|, \phi) = 0$ for $\mathscr{L}^{1}$ almost all $t \in J$.
    Thus, for $\mathscr{L}^{1}$ almost all $t_{1}, t_{2} \in J$, with $t_{1} \leq t_{2}$, we have that $\| \langle V, t_{2} \rangle \| \leq \| \langle V, t_{1} \rangle \|$, which implies that $S (t_{2}) \subset S (t_{1})$, and for $x \in S (t_{2})$, $\theta (t_{2}) (x) \leq \theta (t_{1}) (x)$. 
    The conclusions then follow from \ref{Space time Brakke flow:1}, \ref{Space time Brakke flow:2}, \ref{Space time Brakke flow:3}, \ref{Space time Brakke flow:4}.
\end{proof}

\begin{remark}
    The hypothesis $\delta V = 0$ in \ref{prop: flat Brakke flows} is actually superfluous. 
    Indeed, for $\mathscr{L}^{1}$ almost every $t \in J$, we have that $\langle V , t \rangle \in \mathbf{IV}_{k} (\mathbf{U} (0, R))$, with $\| \delta \langle V, t \rangle \|$ a Radon measure on $\mathbf{U} (0, R)$, such that $\| \delta \langle V, t \rangle \|$  is absolutely continuous with respect to $\| \langle V, t \rangle \|$.
    It then follows from \cite[8.3, 15.6]{Menne-WDFV} (which relies on \cite[Theorem 1]{MR3023856}), that for such $t \in J$, $\delta \langle V, t \rangle = 0$.
    The reason we include the assumption that $\delta V = 0$ is that it in the instances in this paper in which we use \ref{prop: flat Brakke flows}, we are able to deduce that $\delta V = 0$, and thus the results in this paper do not rely on any prior knowledge of \cite{MR3023856} and \cite{Menne-WDFV}.
\end{remark}

\begin{defparagraph}
    Suppose $0<r<\infty$, $V$ is a  Radon measure over $J \times \mathbf{G}_{k} (U)$, and $0<r<\infty$, we denote, for $(t_{0}, x_{0}) \in J \times U$ the rescalings
    \begin{equation*}
        V^{t_0,x_0;r}(\alpha) = \textstyle r^{-k-2}\int \alpha(r^{-2}(t-t_0),r^{-1}(x-x_0),S) \ d V(t,x,S),
    \end{equation*}
    for $\alpha \in \mathscr K(\{ s \, \colon \, t_{0} + r^{2} s \in J \} \times \G_k(\{ y \, \colon \, x_{0} + r y \in U\}))$.
    For $\Omega \subset J \times U$, we denote, 
    \begin{equation*}
        \textstyle \Omega^{t_{0}, x_{0}; r} = \{ (t, x) \, \colon \, (t_{0} + r^{2} t, x_{0} + r x) \in \Omega \}.
    \end{equation*}
\end{defparagraph}

\begin{definition}
    Suppose $V$ is a Radon over $J\times \G_k(U)$, for $(t, x) \in J \times U$ we define
    \[ \VarTan(V,t,x)\]
    to be Radon measures $C$ over $\mathbf{R} \times \mathbf{G}_{k} (\mathbf{R}^{n})$, such that there exists a sequence $r_{i} \searrow 0$, such that, for all $\alpha \in \mathscr{K} (\mathbf{R} \times \mathbf{G}_{k} (\mathbf{R^{n}}))$,  
    \begin{equation*}
         C (\alpha) = \lim_{i \rightarrow \infty} r_{i}^{-k - 2} \int \alpha (r_{i}^{-2} (s - t), r_{i}^{-1} (y - x), S) \, d V_{(s, y, S)}.
    \end{equation*}
\end{definition}

\begin{remark}\label{rmk:almost every where stationary tangent flow}
    If $(t,x)  \in J \times U$ and $\Theta_{\rho}^{\ast k+2}(\|\delta V\|,t,x)<\infty$, then  
    every $C \in \VarTan(V,t,x)$ is stationary, i.e. $\delta C = 0$.
    Note that  $\Theta_{\rho}^{\ast k+2}(\|\delta V\|,t,x)<\infty$ holds $\mathscr{H}^{k + 2}_{\rho}$ almost everywhere (see \cite[2.10.19(3)]{MR41:1976}), and hence by \ref{proppara: a.c. of weight measure and hausdorff measure for Brakke flows}, $\| V \|$ almost everywhere.
\end{remark}

\begin{theorem}[Existence of Tangent Flows \protect{\cite[3.8]{tonegawa2019brakke}}]\label{thm: existence of tangent flows}
    Let $V$ be a space-time-Grassmann Brakke flow on $J \times U$. 
    Then for every $(t_{0}, x_{0}) \in J \times U$, and sequence of positive numbers $r_{i} \downarrow 0$, there exists a subsequence $r_{i_{1}}, \, r_{i_{2}}, \ldots$, and a space-time-Grassmann Brakke flow $\tilde{V}$ on $\mathbf{R} \times \mathbf{G}_{k} (\mathbf{R}^{n})$
    such that, 
    \begin{equation*}
        \tilde{V} = \lim_{j \rightarrow \infty} V^{t_{0}, x_{0}; r_{i_{j}}}
    \end{equation*}
    and satisfies the following three properties,
    \begin{enumerate}[label=\ref{thm: existence of tangent flows}.\arabic{enumi}]
        \item\label{thm: scale invariance of tangent flows} For every $0<r<\infty$,
        \[ \langle \tilde{V}, - r^{2} \rangle^{0; r}  = \langle \tilde{V}, -1 \rangle,\]
        recalling the notation from \ref{proppara: varifold notation}.
        \item\label{thm: Monotonicity of tangent flows} For every $-\infty <t<0$,
        \[ \Theta(V,t_0.x_0) = \textstyle \int \Phi_{0,0}(t,x) \ d \|\langle \tilde{V},t\rangle\|_x.\]
        \item\label{thm: self shrinker equation} For $\tilde{V}$ almost all $(t, x, S)$ with $t<0$,
        \begin{equation*}
        \mathbf{h} (\langle \tilde{V}, t \rangle, x) = \frac{S^{\perp}_{\natural} (x)}{2 t}.
    \end{equation*}
    \end{enumerate}
\end{theorem}

\begin{proof}
By \ref{prop: upper density bounds} and \ref{thm:compactness of Brakke flows}, we have that there exists a subsequence $r_{i_{1}}, \, r_{i_{2}}, \ldots$, and a space-time-Grassmann Brakke flow $\tilde{V}$ over $\mathbf{R} \times \mathbf{G}_{k} (\mathbf{R}^{n})$, such that $\tilde{V} = \lim_{j \rightarrow \infty} V^{t_{0}, x_{0}; r_{i_{j}}}$. 
Proceeding as in the proof of \cite[3.8]{tonegawa2019brakke}, 
there exists $W\in \IVar_k(\mathbf R^n)$ such that for $\mathscr{L}^{1}$ almost all $r < 0$,
\begin{equation*}
    r^{-k} (\boldsymbol{\mu}_{r^{-1}})_{\sharp}\langle \| \tilde{V} \|^{-}, -r^{2} \rangle= \| W\| = r^{-k}(\boldsymbol{\mu}_{r^{-1}})_{\sharp} \langle \| \tilde{V} \|^{+}, - r^{2} \rangle,
\end{equation*}
and hence, by \ref{Space time Brakke flow:1}, and \ref{Space time Brakke flow:2}, one extends this to all $r > 0$.
Moreover, for $\mathscr{L}^{1}$ almost all $r > 0$, 
we have
\begin{equation*}
    \langle \tilde{V}, - r^{2} \rangle = W^{0; r^{-1}},
\end{equation*}
and hence, as $W^{0; r^{-1}}$ is continuous in $r > 0$, $\langle \tilde{V}, t \rangle$ will exist for all $t < 0$, with the above holding for all $r > 0$.
Proceeding as in \cite[3.8]{tonegawa2019brakke} one also deduces \ref{thm: Monotonicity of tangent flows} and \ref{thm: self shrinker equation}.

\end{proof}

\begin{remark}
    By \ref{thm: existence of tangent flows}, for a space-time-Grassamann Brakke flow $V$, we have that for all $(t, x) \in J \times U$, $\mathrm{VarTan} (V, t, x)$ is non-empty. 
\end{remark}

\begin{definition}
    Whenever $n,m \in \mathscr P$, $m\leq n$, $T\in \G(n,m)$, $a\in \mathbf R \times \mathbf R^n$,
    $0<r<\infty$, and $0<s<\infty$, we define 
    \[ X(a,r,\mathbf{R} \times T,s) = (\mathbf R \times \mathbf R^n) \cap \{ b : s^{-1}\dist_{d}(b-a, \mathbf{R} \times T) < \|b-a\| <r\}.\]
\end{definition}
\begin{lemma}[\protect{\cite[3.5(2)]{mattila2021parabolicrectifiabilitytangentplanes}}]
\label{Paraoblic Lipshitz graph condition}
    Suppose, $T \in \G(n, k)$, $\mathbf p_{T}$ and $\mathbf q_{T}$ denote the orthogonal projections on to $\mathbf{R} \times T$ and $\{0\} \times T^{\perp}$, respectively,
    $E\subset \mathbf R \times \mathbf R^n$, 
    $0<r<\infty$, $0<s<1$, and  
    \[ E \cap \mathbf{B}_{d} (a, r) \sim \Clos \, X(a,r, \mathbf{R} \times T, s) = \emptyset \ \text{for all $a\in E$}.\]
    Then for every subset $S \subset E$ with diameter less than or equal to $r$, $\bfp_{T}|S$ is univalent, $\Lip \bfq_{T} \circ (\bfp_{T}|S)^{-1}\leq s (1 - s^{2})^{-1 / 2}$, and 
    $S = \{ x + \bfq_{T}\circ (\bfp_{T} |S)^{-1}(x) : x\in \mathbf{p}_{T} (S)\}$.
\end{lemma}

\begin{proof}
    If $a,b\in S$, $\|a-b\| \leq r$, then 
    $b \in E \cap \mathrm{clos} \, X(a,r, \mathbf{R} \times T,s)$, and thus 
    \[ \|\mathbf{p}_{T}(a-b)\|^2 \geq (1 - s^2) \|a-b\|^2,\]
    from which we infer $\Lip\bfq_{T} \circ (\bfp_{T}|S)^{-1} \leq s(1 - s^{2})^{-1/2}$, and the rest of the statement follows.
\end{proof}

\begin{lemma}\label{flat tangent flow lemma}
  Suppose $V$ is a space-time-Grassmann Brakke flow in $J\times U$, and $(t_{0}, x_{0}) \in \mathrm{spt} \, \| V \|$, is such that, 
  \begin{equation*}
      \Theta^{* k + 2}_{\rho} (\| \delta V \|, t_{0}, x_{0}) < + \infty, 
  \end{equation*}
  $T = \Tan^{k}(\langle V,t_{0}\rangle,x_{0}) \in \G(n,k)$, and for every $0<\varepsilon<\infty$,
  \[ \lim\limits_{r\downarrow 0} \frac{\|V\|(\mathbf B_{\rho}(t_{0}, x_{0}, r)\cap \{ (t, x) : \|\Tan^{k}(\langle V,t\rangle,x)-T\|>\varepsilon\})}{\|V\|(\mathbf B_{\rho}(t_{0}, x_{0}, r))} = 0,\]
  and moreover suppose we have a sequence of positive real numbers $r_1,r_2,\dots$, tending to $0$, such that
  \[ W = \lim\limits_{i\to\infty} V^{t_0,x_0;r_i}.\]
  Then 
  there exists non-increasing functions $\theta^{-}$ and $\theta^{+}$, on $\mathbf{R}$, which are non-negative integer valued, such that
  \[ \langle W^{\pm},t\rangle = \theta^{\pm}(t)\cdot \vVarOp (T) \ \text{ for all $t\in \mathbf R$},\] 
  \[  \theta^{\pm} (t) = \Theta(V,t_0,x_0), \ \text{for all $t<0$},\]
  \[  \theta^{+} (0) \leq \mathbf{\Theta}^{k} (\| \langle V, t_{0} \rangle \|, x_{0}) \leq \theta^{-} (0). \]
\end{lemma}

\begin{proof}
    By \ref{thm: existence of tangent flows} we have that $W \in \mathrm{VarTan} (V, t_{0}, x_{0})$ is a space-time-Grassmann Brakke flow over $\mathbf{R} \times \mathbf{G}_{k} (\mathbf{R}^{n})$, and by \ref{rmk:almost every where stationary tangent flow}, $\delta W = 0$.
    Moreover, one readily verifies that
  \begin{equation*}
    \textstyle \int \varphi (t, x) \| T - S \| \, d W_{(t, x, S)} = \textstyle  \lim\limits_{i\to\infty}\int \varphi(t,x) \| T-S \| \ d V^{t_0,x_0,r_i}_{(t,x,S)} = 0 
\end{equation*}
for every $\varphi \in \mathscr K( \mathbf R \times \G_k(\mathbf R^n))$. 
The conclusions then follow from \ref{prop: flat Brakke flows} and \ref{thm: existence of tangent flows}.
\end{proof}

\begin{lemma}\label{lem:convergence of blow up spt}
  Suppose $V$ is a space-time-Grassmann Brakke flow over $J\times \mathbf{G}_{k} (U)$, and $(t_0,x_0)\in \mathrm{spt} \, \| V \|$, is such that
  \begin{equation*}
      \Theta^{* k+2} (\| \delta V \|, t_{0}, x_{0}) < + \infty,
  \end{equation*}
  $T = \Tan^k(\langle V, t_{0} \rangle, x_0) \in \G(n,k)$, and for every $0<\varepsilon<\infty$,
  \[ \lim\limits_{r\downarrow 0} \frac{\|V\|(\mathbf B_{\rho}(t_{0},x_{0},r)\cap \{ (t, x) : \|\Tan^{k}(\langle V,t\rangle,x)-T\|>\varepsilon\})}{\|V\|(\mathbf B_{\rho}(t_{0}, x_{0}, r))} = 0,\]
  and $(s,y)\notin \mathbf R \times T$.
  Then 
  \[ \liminf\limits_{r\downarrow 0} \dist_{d}((s,y),(\mathrm{spt} \, \| V \|)^{t_0,x_0;r})> 0 .\]
\end{lemma}

\begin{proof}
Let $r_{1}, \, r_{2}, \, \ldots \downarrow 0$ be a sequence of positive real numbers, and choose $R$ so that 
$\mathbf B_{\rho}(s,y, 2R) \cap (\mathbf R \times T) = \emptyset$. 
By \ref{thm: existence of tangent flows}, we may assume that $V^{t_{0}, x_{0}; r_{i}}$ converge as Brakke flows, to a space-time-Grassmann Brakke flow $W$, satisfying the conclusions of \ref{flat tangent flow lemma}.
Taking $i$ large enough such that $\mathbf{B}_{\rho} (s, y,  3R) \subset (J \times U)^{x_{0}, t_{0}; r_{i}}$, we have by \ref{prop: Brakke flow mean value inequality} that if $(\mathrm{spt} \, \| V \|)^{t_{0}, x_{0}; r_{i}} \cap \mathbf{B}_{\rho} (s, y, R) \not= \emptyset$, then, 
\begin{equation*}
    \| V^{t_{0}, x_{0}; r_{i}} \| (\mathbf{B}_{\rho} (s, y,  2R)) \geq c (n, k) R^{k + 2}.
\end{equation*}
However, from \ref{flat tangent flow lemma} we have, $\| W \| (\mathbf{B}_{\rho} (s, y,  2R)) = 0$
and the conclusion follows.
\end{proof}

 \begin{defparagraph}\label{def: notation for static planar Brakke flow}
For every $T\in \G(n,k)$, let $v(J;T)$ be the Brakke flow characterized by 
     \[ v(J ; T) (\alpha) = \textstyle \int_J \alpha(t,x,T) \ d \mathscr H^k\weight T \ d \mathscr L^1_t, \quad \text{for $\alpha\in \mathscr K(J\times \G_k(U))$},\]
     and recalling \ref{para: parabolic metrics} we have that
     \begin{equation*}
          \| v (J; T) \| = \boldsymbol{\beta} (k) \boldsymbol{\omega} (k)^{-1} \mathscr{H}^{k + 2}_{d} \weight \mathbf{R} \times T = 2 \boldsymbol{\alpha} (k) \boldsymbol{\nu} (k)^{-1} \mathscr{H}^{k + 2}_{\rho} \weight \mathbf{R} \times T.
     \end{equation*}
 \end{defparagraph}

\begin{theorem}
\label{thm:rectifiability of Brakke flow}
    Suppose $V$ is a space-time-Grassmann Brakke flow over $J\times \mathbf{G}_{k} (U)$.
    Then, $\spt \|V\|$ is a (vertical) parabolic ($k+2$)-rectifiable set in $J \times U$.
\end{theorem}

\begin{proof}

Let $T_1,T_2,\dots$ be a dense sequence in $\G(n,k)$.
Let $\rho$ and $C$ be as in \ref{proppara: parabolic Besicovitch}, and $N$ be the set of $(t,x) \in J \times U$ at which $\mathrm{Tan}^{k} (\| \langle V, t \rangle \|, x)$ is not $(\|V\|,C)$ approximately continuous, or $\Theta^{* k + 2} (\| \delta V \|, t, x) = + \infty$. 
Then, from \cite[Theorem C]{LW-Space-Time-GrassmannMeasureBrakkeFlow}, \ref{proppara: approximate continuity}, and \ref{rmk:almost every where stationary tangent flow}, $\|V\|(N) = 0$ , and then by \ref{proppara: a.c. of weight measure and hausdorff measure for Brakke flows}, $\mathscr H^{k+2}_{d}(N \cap \mathrm{spt} \, \| V \|) = 0$ (equivalently $\mathscr{H}^{k+2}_{\rho} (N \cap \mathrm{spt} \, \| V \|) = 0$).
    Fix $0 < s < 1$, and denote $E = \spt \|V\| \sim N$, and
    \[ E_{i, j} = E \cap  \{  a : E \cap \mathbf{B}_{d} (a, 1 / j)\sim  \Clos X(a,1 / j, \mathbf{R} \times T_{i} ,s) = \emptyset  \}.\]
    We will show that by \ref{lem:convergence of blow up spt}, that $E = \bigcup_{i, j=1}^\infty E_{i, j}$.
    Indeed, if $a = (t, x)\in E$ and we take $T_{i}$ such that $\| T - T_{i}\| \leq s / 2$, for $T = \mathrm{Tan}^{k} (\| \langle V, t \rangle \|, x)$, then if $a \not\in E_{i, j}$ for all $j = 1, \, 2, \ldots$, there exists a sequence of points 
    \begin{equation*}
        p_{j} = (t_{j}, x_{j}) \in E \cap \mathbf{B}_{d} (a, 1 / j) \sim \Clos \, X (a, 1 / j, \mathbf{R} \times T, s / 2) \subset \mathrm{spt} \, \| V \|.
    \end{equation*} 
    Denoting $\lambda_{j} = \| a - p_{j} \| > 0$, we may assume $v_{j} = (\lambda_{j}^{-2} (t_{j} - t), \lambda_{j}^{-1} (x_{j} - x)) \rightarrow v$, with $v \notin \mathbf{R} \times T$, contradicting \ref{lem:convergence of blow up spt}.
    By \ref{Paraoblic Lipshitz graph condition}, 
    we infer that $\spt \|V\|$ is a (vertical) parabolic $(k + 2)$-rectifiable set.
\end{proof}

\begin{theorem}\label{thm: existence of unique static tangent flow}
    Suppose $V$ is Brakke flow over $J \times \G_k(U)$. Then 
    for $\|V\|$ almost every $(t,x)\in J\times U$,
    $\VarTan(V,t,x)$ consists of the single element
        \begin{equation}
            \label{eq:tangent flow eq}
                \textstyle \Density^{k}(\| \langle V, t \rangle \|, x)\cdot v(\mathbf R;T),
        \end{equation} 
        where $T = \Tan^k(\langle V,t\rangle,x)$, and moreover, 
        \begin{eqnarray*}
            \Theta (V, t, x) = \mathbf{\Theta}^{k} (\| \langle V, t \rangle \|, x) &=& (2 \boldsymbol{\alpha} (k))^{-1} \Theta^{k + 2}_{\rho} (\| V \|, t, x), \\ 
            &=& \boldsymbol{\beta} (k)^{-1} \Theta^{k + 2}_{d} (\| V \|, t, x).
        \end{eqnarray*}
\end{theorem}

\begin{proof}
    For $\|V\|$ almost every $(t, x) \in J \times U$, the following three statements hold,
    \begin{enumerate}
        \item $\Theta^{* k+2} (\| \delta V \|, t, x) < + \infty$, see \ref{rmk:almost every where stationary tangent flow}.
        \item $\Tan^k(\langle V,t\rangle,x)$ is $(\| V \|, C)$ approximately continuous at $(t, x)$, coming from \cite[Theorem C]{LW-Space-Time-GrassmannMeasureBrakkeFlow}, and \ref{proppara: approximate continuity}.
        \item Tangent measures of $\|V\|$ at $(t, x)$ have the form
    \[ \{ c \cdot \| v (\mathbf{R}; \mathrm{Tan}^{k} (\| \langle V, t \rangle\|, x))\| : 0<c< \infty\},\]
    coming from \ref{thm:rectifiability of Brakke flow}, \ref{thm: rectifiability thm from Mattila}, \ref{proppara: a.c. of weight measure and hausdorff measure for Brakke flows}, \ref{lem: equality of tangent measure for a.c. measures}, and \ref{def: notation for static planar Brakke flow}.
    \end{enumerate}
    Then by \ref{flat tangent flow lemma}, if $W \in \mathrm{VarTan} (V, t, x)$, we must have that, 
    \begin{equation*}
        W = c \cdot v (\mathbf{R}; \mathrm{Tan}^{k} (\| V \|, x)),
    \end{equation*}
    with $c$ given by, 
    \begin{equation*}
        c = \Theta (V, t, x) = \mathbf{\Theta}^{k} (\| \langle V, t \rangle \|, x).
    \end{equation*}
    Moreover, we must also have that
    \begin{equation*}
        c = (2 \boldsymbol{\alpha} (k))^{-1} \Theta_{\rho}^{k + 2} ( \| V \|, t, x) = \boldsymbol{\beta} (k)^{-1} \Theta_{d}^{k + 2} (\| V \|, t, x). \qedhere
    \end{equation*}
\end{proof}

\begin{corollary}
    Suppose $V$ is a space-time-Grassmann Brakke flow over $J \times \mathbf{G}_{k} (U)$, then, 
    \begin{equation*}
        \| V \| = \mathscr{H}^{k + 2}_{\rho} \weight \boldsymbol{\nu} (k)^{-1} \Theta_{\rho}^{k + 2} (\| V \|, \, \cdot \, ) = \mathscr{H}^{k + 2}_{d} \weight \boldsymbol{\omega} (k)^{-1} \Theta_{d}^{k + 2} (\| V \|, \, \cdot \, ).
    \end{equation*}
\end{corollary}

\begin{proof}
    By \ref{thm:rectifiability of Brakke flow}, \ref{thm: mattila existence of density a.e. for parabolic rectifiable sets}, and \ref{proppara: a.c. of weight measure and hausdorff measure for Brakke flows}, we have that for $\| V \|$, $\mathscr{H}^{k + 2}_{\rho}$, and $\mathscr{H}^{k + 2}_{d}$, almost all $(t, x) \in \mathrm{spt} \, \| V \|$, 
    \begin{equation*}
        \lim_{r \downarrow 0} r^{- (k+2)} \mathscr{H}_{\rho}^{k + 2} (\mathrm{spt} \, \| V \| \cap \mathbf{B}_{\rho} (t, x, r)) = \boldsymbol{\nu} (k), 
    \end{equation*}
    and, 
    \begin{equation*}
        \lim_{r \downarrow 0} r^{- (k+2)} \mathscr{H}_{d}^{k + 2} (\mathrm{spt} \, \| V \| \cap \mathbf{B}_{\rho} (t, x, r)) = \boldsymbol{\beta} (k)^{-1} 2 \boldsymbol{\alpha} (k) \boldsymbol{\omega} (k).
    \end{equation*}
    The conclusion then follows from \ref{propara: differentiation of Radon measures}, \ref{proppara: a.c. of weight measure and hausdorff measure for Brakke flows}, and \ref{thm: existence of unique static tangent flow}.
\end{proof}

\bibliographystyle{plain}
\bibliography{Parabolic-Rectifiability-BrakkeFlow}

\end{document}